\documentclass[11pt]{amsart}
\usepackage{amsmath,amsthm,amssymb}
\usepackage{orcidlink}
\usepackage{hyperref}
\usepackage{color}
\usepackage{graphicx}
\usepackage{blindtext}
\numberwithin{equation}{section}
\newtheorem{theorem}{Theorem}[section]
\newtheorem{lemma}[theorem]{Lemma}

\newtheorem{proposition}[theorem]{Proposition}

\theoremstyle{definition}
\newtheorem{definition}[theorem]{Definition}

\theoremstyle{remark}
\newtheorem{remark}[theorem]{Remark}

\usepackage{enumerate}
\usepackage[margin=3cm]{geometry}

\usepackage{orcidlink}

\newcommand{\norm}[1]{\left\lVert #1\right\rVert}

\newcommand{\R}{\mathbb{R}}
\newcommand{\C}{\mathbb{C}}

\title{Sharp decay rate and asymptotic simplification for a quasilinear heat equation}

\author{
Maryam Al Hajjar\orcidlink{0009-0000-5378-0301}
\and
Ioana Ciotir\orcidlink{0000-0002-8831-050X}
\and
Matthias T\"aufer\orcidlink{0000-0001-8473-2310}
}

\address{Maryam Al Hajjar, CERAMATHS, Université Polytechnique Hauts-de-France,
Le Mont Houy 59313 Valenciennes Cedex 09,
France; \text{Maryam.Hajjar@uphf.fr}.}
\address{Ioana Ciotir, Normandie University, INSA de Rouen Normandie, LMI (EA 3226 – FR CNRS 3335), 76000 Rouen, France; \text{ioana.ciotir@insa-rouen.fr}.}
\address{Matthias Täufer, CERAMATHS, Université Polytechnique Hauts-de-France,
Le Mont Houy 59313 Valenciennes Cedex 09,
France;
\text{Matthias.Taufer@uphf.fr}.
}

\thanks{Acknowledgement: The research of the first and third named author was supported by Agence Nationale de Recherche under the grant Modélisation mathématique et optimisation de la Fabrication Additive céramique (Chaire de Professeur Junior), and by COST (European Cooperation in Science and Technology) through COST Action 24122 mSPACE, www.cost.eu.
}
\title{Sharp decay rate and asymptotic simplification for a quasilinear heat equation}

\begin{document}

\begin{abstract}

We study the decay rate for a quasilinear heat equation of the form $ \dot u-\Delta u =  \operatorname{div  B}(\nabla u)$. 
The term on the right-hand-side is interpreted as a nonlinear perturbation of the ordinary heat equation and we will prove an optimal decay rate for the solution. More precisely, under natural conditions on the operator $B \colon \R^d \to \R^d$, we show that for sufficiently regular initial data  
the solution decays as $t^{-d/4}$ and its gradient as $t^{-d/4 - 1/2}$ for large times, that is at the same rates as for the free heat equation.
Furthermore, we prove that the difference between the quasilinear and the linear solution decays \emph{strictly faster} than the solution of either equation generically does, that is that there is \emph{asymptotic simplification}.
All estimates are completely explicit and rely on variations of Fourier splitting techniques, originally introduced by Schonbek.
\end{abstract}
\maketitle

\bigskip

\noindent\textbf{Keywords:} Quasilinear heat equation, Duhamel formula, decay rate for the solution, decay rate for the error. \\
\noindent\textbf{MSC2020 Classification:} {
47H06, 
35A01, 
35B40. 
}

\bigskip

\section{Introduction}

This article is about decay rates of solutions of quasilinear heat equations in full space and about asymptotic simplification towards solutions of the corresponding linear equation. 

The Fourier splitting technique was introduced to nonlinear evolution equations by Schonbek  in 1985 to prove sharp decay rates for Navier-Stokes equations~\cite{Schonbek-85}.
Since then, this method has been taken up by numerous authors for a large variety of models see, e.g. \cite{Schonbek-86, Sciacca2017, Li2023, Kosloff2024, Ikeda2026} .
We mention here~\cite{Yang-14} on decay for equations of the form
\[
    \partial_t u - \Delta u + \lvert \nabla u \rvert^2 u = 0
\]
which is probably closest to our first result, Theorem~\ref{1stthm}.

A parallel development in recent decades has been the question of \emph{asymptotic simplification}, that is the whether solutions of nonlinear equations will tend towards solutions of a simpler, linear equation, see~\cite{EscobedoZ-91} for nonlinearities polynomial in $u$ and~\cite{CarrilloGCYZ-23} for more recent developments.

In this paper, we contribute to this literature by considering a quasilinear
heat equation with a gradient-dependent nonlinearity in divergence form. Our
main results establish that, under natural monotonicity and Lipschitz
assumptions on the nonlinearity, the solution decays at the same optimal
$L^2$ rate as the linear heat equation, and that the nonlinear remainder
decays strictly faster. This provides a quantitative form of asymptotic
simplification for a class of quasilinear problems that has not been
previously treated with explicit rates.

In fact, we consider a quasilinear heat equation of the form
\begin{equation} \label{system1}
\begin{cases}
        \partial_tu-\Delta u =  \operatorname{div  B}(\nabla u)\\
       u(0) =u_0,
 \end{cases}    
\end{equation}
where $u_0 \in L^1( \R^d) \cap H^2(\R^d)$ and $\operatorname{B}: \R^d \to \R^d$ is a nonlinear monotone operator.
Heuristically, if $\operatorname{B}(0) = 0$ and $\operatorname{B}$ is sufficiently regular near $0$, e.g.
some polynomial behaviour of sufficiently high degree, we expect higher
frequencies to dissipate and the solution $u(t)$ to behave like the solution of the free heat equation $\mathrm{e}^{t \Delta}u_0$.
Our work will make this intuition precise.

Recall that the solution $\mathrm{e}^{t \Delta} u_0$ of the linear heat equation satisfies the decay rate 
\[
 \lVert \mathrm{e}^{t \Delta} u_0 \rVert _{L^2(\R^d)} \leq C t^{-d/4} \lVert u_0 \rVert_{L^1(\R^d)}
 \quad
 \text{and}
 \quad
  \lVert \nabla \mathrm{e}^{t \Delta} u_0 \rVert _{L^2(\R^d)} \leq C t^{-d/4 - 1/2} \lVert u_0 \rVert_{L^1(\R^d)}.
\]
In this article, we prove the same asymptotic decay rates for the solution of the quasilinear heat equation \eqref{system1}. 
Furthermore, we prove that the relative error
\[
 u(t)-\mathrm{e}^{t \Delta }u_0
\]
decays at a strictly faster rate, that is, as $t$ tends to $\infty$, the solution of the quasilinear problem will be asymptotically close to the solution of the linear one.

One important motivation for this paper arises from modeling and simulation associated with modern manufacturing techniques such as laser ablation and selective laser sintering (SLS) and selective laser melting (SLM).
For instance, in~\cite{Bauer2015}, a process of laser ablation is investigated, where ultra-short laser pulses are used to evaporate parts of a steel surface. 
The laser pulses will deposit residual heat in the metal which can lead to undesirable effects of surface degradation.
In the above article, the authors conducted a numerical simulation taking the (nonlinear) heat-dependent thermal properties of the material into account.
While this choice of a nonlinear heat equation seems reasonable in a regime of short times (i.e. during and around pulses) where violent temperature gradients and the multi-physics nature of the process will actually require a more comprehensive formulation~\cite{Omeaca2024}, other research~\cite{MercelisK-06} suggests that the surface degradation might be due to effects on larger scales in time and space; in particular heat accumulation in the bulk of the piece.
So, understanding the question of asymptotic simplification in this context can shed light on the question in which scenarios a simpler linear heat equation simulation would be justified.
Similar questions arise in the context of selective laser melting where parts of a granular material are selectively solidified under the influence of a laser.
Here, questions of path optimization have been investigated~\cite{Alam2021, Hmede2025, Hmede2026}, mostly in order to minimize thermal stresses.
Since the associated path-optimization algorithms require multiple iterative simulations, the question of which simplifications to the underlying model are physically justified, is particularly relevant with asymptotic simplification making a case for using a simpler, linear model.

The organisation of the paper is the following. 
Section~\ref{sec:preliminary} collects some facts on decay rates of the linear heat equation as a baseline case. 
Section~\ref{sec:existence} is about existence of solutions of our equation and justifies the use of a Duhamel formula - a central ingredient in the subsequent proofs.
Section~\ref{sec:results} formulates our main results while Section~\ref{sec:proofs} contains their proofs.

 \subsection{Notation and preliminaries}
    \label{sec:preliminary}
    Throughout this article, we denote by $L^p(\R^d)$, $p \in [1, \infty)$, the usual Lebesgue spaces and by $H^1(\R^d)$ and $H^2(\R^d)$ the first- and second order Sobolev spaces in $L^2(\R^d)$, that is the spaces of functions all whose weak derivatives up to first or second order, respectively, exist as functions in $L^2(\R^d)$. 
    By $\Delta = \frac{\partial^2}{\partial_{x_1}^2} + \dots + \frac{\partial^2}{\partial_{x_d}^2}$, we denote the $d$-dimensional Laplace operator, and by $\partial_t := \frac{\partial}{\partial_t}$, the differentiation with respect to the time variable $t$.
    By $(S(t))_{t \geq 0} = (\mathrm{e}^{t \Delta})_{t \geq 0}$, we denote the heat semigroup.
Given a vector-valued function $B \colon \R^d \to \R^d$, we denote by $(JB) \colon \R^d \to \R^{d \times d}$ its \emph{Jacobian}.
Given a matrix $A \in \R^{d \times d}$, we denote by $\lVert A \rVert_F^2 = \sum_{i,j = 1}^d \lvert A_{ij} \rvert^2$ its \emph{Frobenius norm}.
Furthermore, given $u \colon \R^d \times \R \to \C$, we denote by $\mathrm{Hess}(u) (x,t) = \left(\frac{\partial^2}{\partial_{x_i} \partial_{x_j}} u(x,t) \right)_{i,j = 1}^d$ its \emph{Hessian in the space variables}. 
    
\    
 
 Let us recall known decay estimates for the solution $\mathrm{e}^{t \Delta}u_0$ of the linear heat equation
 \[
    \partial_t u - \Delta u = 0,
    \quad
    u(0) = u_0
 \]
 and its derivatives as a baseline to compare subsequent results with.
\begin{proposition}
    \label{prop:decay_1}
    For every $1 \leq q \leq p \leq \infty$, every $t>0$ and $f \in L^p(\R^d) \cap L^q(\R^d)$, we have
    \begin{equation}
      \displaystyle  \lVert \mathrm{e}^{t \Delta }f\rVert_{L^p(\R^d)} \leq C t^{-\frac{d}{2}\left( \frac{1}{q} - \frac{1}{p}\right)} \lVert f \rVert_{L^q(\R^d)}
    \end{equation}
    and
        \begin{equation}
      \displaystyle  \lVert \nabla \mathrm{e}^{t \Delta }f\rVert_{L^p(\R^d)} \leq C t^{-\frac{d}{2}\left( \frac{1}{q} - \frac{1}{p}\right)-\frac{1}{2}} \lVert f \rVert_{L^q(\R^d)},
    \end{equation}
    where $C $ is a positive constant independent of $t$ and $f$.
\end{proposition}

Proposition~\ref{prop:decay_1} is proved by combining the expression $\mathrm{e}^{t \Delta} f = K_t \ast f$, where 
\[
    K_t(x)
    =
    \frac{1}{(4 \pi t)^{d/2}}
    \exp 
    \left(
        -\frac{\lvert x \rvert^2}{4t}
    \right)
\]
is the heat kernel in $\R^d$, with Young's convolution inequality.
Specializing to the case $p = 2$ and $q = 1$, we obtain:

\begin{proposition}
    For $u_0 \in L^2(\R^d) \cap L^1(\R^d)$, we have, for every $t>0,$
    \[
    \lVert  \mathrm{e}^{t \Delta } u_0 \rVert _{L^2(\R^d)} \leq C t^{-d/4} \lVert u_0 \rVert_{L^1(\R^d)},
    \]
    and
    \[
    \lVert \nabla \mathrm{e}^{t \Delta } u_0 \rVert _{L^2(\R^d)} \leq C t^{-d/4 - 1/2} \lVert u_0 \rVert_{L^1(\R^d)}.
    \]
\end{proposition}

Next, let us show that these decay rates are sharp.

\begin{proposition}
    \label{prop:lower_bound}
    Let $u_0 \in L^2(\R^d) \cap L^1(\R^d)$ with $\int_{\R^d} u_0(x)\mathrm{d}x \neq 0$. Then, there is a constant $c > 0$ such that for $t \geq 1$, one has
    \[
    \lVert \mathrm{e}^{t \Delta } u_0 \rVert_{L^2(\R^d)}
    \geq
    c
    t^{-d/4},
    \]
    and
    \[
    \lVert \nabla \mathrm{e}^{t \Delta } u_0 \rVert_{L^2(\R^d)}
    \geq
    c
    t^{-d/4 - 1/2}.
    \]
\end{proposition}

\begin{remark}
    The condition $t \geq 1$ can clearly be replaced by $t \geq t_0$ for any $t_0 > 0$ upon changing the constant $c$.
\end{remark}

\begin{proof}[Proof of Proposition~\ref{prop:lower_bound}]
    We only prove the first inequality since the second one follows from a similar calculation. Since $u_0 \in L^1(\R^d)$, its Fourier transform is continuous. 
    Furthermore, by $\int_{\R^d} u_0(x)\mathrm{d}x \neq 0$, we have $\hat u_0(0) \neq 0$.
    Consequently, there are $c_0,\delta > 0$ such that $\lvert \hat u_0(\xi) \rvert \geq c_0$ for $\lvert \xi \rvert \leq \delta$.
    We estimate using Plancherel's identity
    \[
    \lVert \mathrm{e}^{t \Delta } u_0 \rVert_{L^2(\R^d)}^2
    \geq
    \int_{\lvert \xi \rvert \leq \delta}
    \mathrm{e}^{-2 t \lvert \xi \rvert^2 }
    \lvert \hat u_0(\xi) \rvert^2 \mathrm{d} \xi
    \geq
    c_0^2
    \int_{\lvert \xi \rvert \leq \delta}
    \mathrm{e}^{-2 t \lvert \xi \rvert^2 }
    \mathrm{d} \xi
    =
    c_0^2
    t^{- d/2}
    \int_{\lvert \eta \rvert \leq \delta \sqrt{t}}
    \mathrm{e}^{-2\lvert \eta \rvert^2 }
    \mathrm{d} \eta.
    \]
    But since $t \geq 1$, the last integral is bounded from below by a positive constant.
\end{proof}

\section{Existence, uniqueness and Duhamel form of the solution}

\label{sec:existence}

From now on, we consider the following Cauchy problem
\begin{equation} \label{system}
\begin{cases}
        \partial_tu-\Delta u =  \operatorname{div  B}(\nabla u)\\
       u(0) =u_0,
 \end{cases}    
\end{equation}
where $u_0 \in L^1( \R^d) \cap H^2( \R^d)$  and $\operatorname{B} \colon \R^d \to \R^d$ is a nonlinear monotone operator.

\begin{definition}
	\label{def:monotone}
 	We call a function $B  \colon \R^d \to \R^d$ \emph{monotone} if for all $v,w \in \R^d$, we have
	\[
	\left \langle B(v)-B(w),v-w \right \rangle \geq 0.
	\] 
    We call it \emph{globally Lipschitz continuous} with Lipschitz constant $L$, if each $B_i \colon \R^d \to \R, \, i=1,..., d$, is globally  Lipschitz continuous with Lipschitz constant $L$.
\end{definition}
In particular, Lipschitz continuity of $B$ with Lipschitz constant $L$ implies that, for every $ i=1,..., d$, $B_i$ is differentiable almost everywhere and its partial derivative is bounded by $L$.

We can prove now the following existence result.

\begin{proposition}
	Let $\operatorname{B}: \R^d \to \R^d$ be monotone and continuous with $B(0)=0$, and let $u_0 \in H^2( \R^d)$.
    Then, there exists a unique strong solution of the system \eqref{system} satisfying the Duhamel formula
    \begin{equation}
    \label{eq:Duhamel}
    u(t)=S(t)u_0+ \int_0^t S(t-s)  \operatorname{div B}(\nabla u(s)) \mathrm{d}s,
    \end{equation}
     where $S(t)u_0=\mathrm{e}^{t \Delta } u_0.$
\end{proposition}

\begin{proof}

\textbf{Step I: Strong solution}
\\
We first prove that \eqref{system} has a unique strong solution in the sense of \cite[Definition 4.1]{Barbu2010}. 
Indeed, by classical theory, the Laplace operator 
\begin{equation*}
	-\Delta 
	\colon 
	L^{2}( \R^d)
	\supset
	\mathcal{D}\left( - \Delta \right) = H^2(\R^d)
	\rightarrow 
	L^{2}(\R^d) 
\end{equation*}
satisfies:
\begin{enumerate}[(i)]
\item \emph{Monotonicity}, that is
\[
\left\langle \left( -\Delta \right) x-\left( -\Delta \right)
y,x-y\right\rangle _{L^{2}\left( \mathbb{R}^{d}\right) }=\int_{\mathbb{R}
^{d}}\left\vert \nabla \left( x-y\right) \right\vert ^{2}\mathrm{d}\xi \geq 0,
\quad
\text{for all $x,y\in H^2(\R^d)$.}
\]

\item \emph{Maximality}, that is
\[
\operatorname{Ran} \left( \operatorname{Id}-\Delta \right) =L^{2}( \R^d) ,
\]
where $\operatorname{Ran}$ denotes the range of an operator. 

\item \emph{Continuity}, that is
\[
\left\Vert \Delta x\right\Vert _{L^{2}\left( \mathbb{R}^{d}\right) }\leq
C\left\Vert x\right\Vert _{H^{2}\left( \mathbb{R}^{d}\right) ,}
\]
for a constant $C$. This follows from the definition of the Sobolev norm
\[
\left\Vert x\right\Vert _{H^{2}\left( \mathbb{R}^{d}\right) }^{2}=\left\Vert
x\right\Vert _{L^{2}\left( \mathbb{R}^{d}\right) }^{2}+\sum_{\left\vert
\alpha \right\vert =1}\left\Vert \partial ^{\alpha }x\right\Vert
_{L^{2}\left( \mathbb{R}^{d}\right) }^{2}+\sum_{\left\vert \alpha
\right\vert =2}\left\Vert \partial ^{\alpha }x\right\Vert _{L^{2}\left( 
\mathbb{R}^{d}\right) }^{2},
\]
through the relation 
\[
\left\Vert \Delta x\right\Vert _{L^{2}\left( \mathbb{R}^{d}\right) }\leq
\sum_{i=1}^{\infty }\left\Vert \partial _{ii}x\right\Vert _{L^{2}\left( 
\mathbb{R}^{d}\right) }\leq \left\Vert x\right\Vert _{H^{2}\left( \mathbb{R}
^{d}\right) }.
\]
\end{enumerate}
In particular, $- \Delta$ is \textbf{maximal monotone} and
 \textbf{continuous}.
On the other hand, the nonlinear operator
\begin{eqnarray*}
T &=&-\operatorname{div} B\left( \nabla \cdot \right) \colon L^2(\R^d) \supset H^2(\R^d) 
\rightarrow L^{2}(\R^d)  \\
D\left( T\right)  &=&H^{2}(\R^d)
\end{eqnarray*}
satisfies:

\begin{enumerate}[(i)]
\item \emph{Monotonicity}, that is
\[
	\left\langle 
	T\left( x\right) -T\left( y\right) ,x-y\right\rangle_{L^{2}\left( \mathbb{R}^{d}\right) }
	=
	\int_{\mathbb{R}^{d}}\left( B\left(
	\nabla x\right) -B\left( \nabla y\right) \right) 
	\left( 
		\nabla x
		-
		\nabla y
	\right) \mathrm{d}\xi \geq 0
	\quad
	\text{for all $x,y\in D\left( T\right)$}.
\]

\item \emph{Hemicontinuity}, that is
\[
t\longmapsto \left\langle T\left( x+ty\right) ,z\right\rangle _{L^{2}\left( 
\mathbb{R}^{d}\right) }
\]
is continuous for all fixed $x,y \in D(T)$. 
Indeed, this follows from continuity of $B$.

\item \emph{Coercivity} of $\operatorname{Id} + T$ in $L^2(\R^d)$, since
\begin{eqnarray*}
\left\langle \left( \operatorname{Id}+T\right) \left( x\right) ,x\right\rangle _{L^{2}\left( 
\mathbb{R}^{d}\right) } &=&\left\Vert x\right\Vert _{L^{2}\left( \mathbb{R}%
^{d}\right) }^{2}+\int_{\mathbb{R}^{d}}B\left( \nabla x\right) \cdot \nabla
x\mathrm{d}\xi  \\
&\geq &\left\Vert x\right\Vert _{L^{2}\left( \mathbb{R}^{d}\right) }^{2}
\end{eqnarray*}
where we used monotonicity of $B$ and the fact that $B(0) = 0$.
\end{enumerate}
By the Browder-Minty theorem, $\operatorname{Id}+T$ is surjective in $L^{2}\left( \mathbb{R}^{d}\right)$, that is $\operatorname{Ran} ( \operatorname{Id}+T ) 
	=
	L^{2}(\R^d)$.
Together with monotonicity of $B$, we obtain that $T$ is \textbf{maximal monotone}.
Now, \cite[Corollary 2.1]{Barbu2010} implies that the sum operator $- \Delta + T$
is also maximal monotone. 
Since $L^2\left( \mathbb{R}^{d}\right) 
$ is a reflexive Banach space, \cite[Theorem 4.5]{Barbu2010} implies that \eqref{system} has a unique strong solution.

\bigskip

\noindent
\textbf{Step II: Duhamel formula}
\\
Next, we prove that the solution of \eqref{system} satisfies a
Duhamel type formula.
Let
\begin{equation}
	\label{eq:v}
v\left( t\right) :=S\left( t\right) u_{0}+\int_{0}^{t}S\left( t-s\right) 
\operatorname{div}B\left( \nabla u\left( s\right) \right) \mathrm{d}s.
\end{equation}
Applying $\partial _{t}$ to~\eqref{eq:v} yields
\begin{equation}
	\label{eq:Duhamel_proof_1}
\partial _{t}v\left( t\right) =\partial _{t}S\left( t\right) u_{0}+\partial
_{t}\int_{0}^{t}S\left( t-s\right) \operatorname{div}B\left( \nabla u\left( s\right)
\right) \mathrm{d}s,
\end{equation}
and applying $\Delta $ to~\eqref{eq:v} implies
\begin{equation}
	\label{eq:Duhamel_proof_2}
\Delta v\left( t\right) =\Delta S\left( t\right) u_{0}+\int_{0}^{t}\Delta
S\left( t-s\right) \operatorname{div}B\left( \nabla u\left( s\right) \right) \mathrm{d}s.
\end{equation}
Taking the difference of~\eqref{eq:Duhamel_proof_1} and~\eqref{eq:Duhamel_proof_2}, and keeping in mind that $S\left( t\right) $ is
the semigroup generated by the Laplace operator, we find
\begin{align*}
\partial _{t}v\left( t\right) -\Delta v\left( t\right)  
&=\partial _{t}\int_{0}^{t}S\left( t-s\right) \operatorname{div}B\left( \nabla
u\left( s\right) \right) \mathrm{d}s-\int_{0}^{t}\Delta S\left( t-s\right) \operatorname{div}
B\left( \nabla u\left( s\right) \right) \mathrm{d}s \\
&=\operatorname{div}B\left( \nabla u\left( t\right) \right) .
\end{align*}
Subtracting~\eqref{system} from the previous equation yields
\[
\left\{ 
\begin{array}{l}
\partial _{t}\left( v\left( t\right) -u\left( t\right) \right) -\Delta
\left( v\left( t\right) -u\left( t\right) \right) =0 \\ 
v\left( 0\right) =u\left( 0\right) =u_{0},
\end{array}
\right. 
\]
which implies that $v\left( t\right) =u\left( t\right) $ and completes the proof.
\end{proof}

\section{The Main Results}

\label{sec:results}

Now we can state the three main theorems of this paper. They provide optimal decay rates for the solution of~\eqref{system} and a bound on the difference which tends to zero strictly faster.

\begin{remark}
\label{rem:Jacobian_implies_quadratic}
    In the three subsequent theorems, we will assume that the Jacobian of $B$ satisfies 
\begin{equation}
    \label{eq:condition_on_Jacobian}
    \lVert (JB)(v) \rVert_F \leq c \lvert v \rvert
    \quad
    \text{for almost all}
    \quad
    v \in \mathbb{R}^d
\end{equation}
where $\lVert \cdot \rVert_F$ denotes the Frobenius norm.
Indeed, since all norms on a finite dimensional vector space are equivalent, this could be replaced by any other norm but the Frobenius norm will simplify some subsequent calculations.
Let us also note that together with $B(0) = 0$, and up to changing the constant $c$, Assumption \eqref{eq:condition_on_Jacobian} implies in particular that \emph{$B$ grows at most quadratically around zero}, that is
\[
    \lvert B(v) \rvert \leq c \lvert v \rvert^2
    \quad
    \text{for almost all}
    \quad
    v \in \mathbb{R}^d,
\]
an immediate consequence of
\[
    B(v)
    =
    B(0)
    +
    \int_0^1
    (JB)(tv) \cdot v
    \
    \mathrm{d} t.
\]
This will be used in the proof of Theorem~\ref{thm:decay_error}.
\end{remark}

\begin{theorem} \label{1stthm}
    Assume that $\operatorname{B}$ is monotone, globally Lipschitz continuous with Lipschitz constant $L < 1$,  $B(0) = 0$, and let $\lVert (JB)(v) \rVert_{F} \leq c \lvert v \rvert$, for almost every $ v \in \R^d$. 
    Let $u_0 \in L^1(\R^d) \cap H^2(\R^d)$.
    Then, for every $t>0$, the solution of \eqref{system} verifies
    \[
     \lVert u(t) \rVert  _{L^2(\R^d)} \leq C_1 (1+t)^{-d/4}
    \]
    where $C_1=\displaystyle \left[ \frac{2\operatorname{vol} S_{d-1}C_3^2d^{d/2}}{d[2(1-L)]^{d/2}} \right]^{1/2} $,  $\operatorname{vol} S_{d-1}=\frac{2 \pi^{d/2}}{\Gamma(d/2)}$ is the area of the unit sphere in $\R^d$  and $C_3 $ will be given in Lemma \ref{2ndlemma}.
\end{theorem}

\begin{theorem} \label{decaynabla}
        Assume that $\operatorname{B}$ is monotone, globally Lipschitz continuous with Lipschitz constant $L < 1$,  $B(0) = 0$, and $\lVert (JB)(v) \rVert_{F} \leq c \lvert v \rvert$, for almost every $ v \in \R^d$.
        Let $u_0 \in L^1(\R^d) \cap H^2(\R^d)$.
        Then, for sufficiently large $t > 0$, the solution of \eqref{system} verifies
     \begin{equation*}
         \lVert \nabla u(t) \rVert_{L^2(\R^d)}\leq C_2(1+t)^{-d/4-1/2},
     \end{equation*} 
     where $C_2= \displaystyle \left[\frac{2\operatorname{vol} S_{d-1}C_3^2(d+1)^{2+d/2}}{d^2[2(1-L)]^{1+d/2}} \right]^{1/2}$ and $C_3 $ will be given in Lemma \ref{2ndlemma}. 
\end{theorem}

Next, we show that the error $r(t) =u(t)-\mathrm{e}^{t\Delta}u_0$ goes to zero \emph{faster than} $\mathrm{e}^{t \Delta} u_0$ and $u(t)$. 

\begin{theorem}
    \label{thm:decay_error}
      Assume that $\operatorname{B}$ is monotone, globally Lipschitz continuous with Lipschitz constant $L < 1$, $B(0) = 0$,  and $\lVert (JB)(v) \rVert_{F} \leq c \lvert v \rvert$, for almost every $ v \in \R^d$.
      Let $u_0 \in L^1(\R^d) \cap H^2(\R^d)$.
        Then, for sufficiently large $t > 0$, the difference $r(t) := u(t) - \mathrm{e}^{t\Delta }u_0$, where $u(t)$ is the solution of \eqref{system}, verifies
      \[
      \lVert r(t) \rVert_{L^2(\R^d)}
      \leq C_4
      t^{- d/4 - 1/(d+4)},
      \]
      where $C_4=\max \{2^{d/4+3/2}CLC_2,\frac{Cc\lVert u_0 \rVert ^2}{2(1-L)} \}$.
      In particular,
      \[
       \lim_{t \to \infty}\lVert r(t) \rVert_{L^2(\R^d)}t^{d/4}=0.
      \]
\end{theorem}

\begin{remark}
	\label{rem:dropping_assumptions}
	It will be clear from the proofs below that the monotonicity of $B$ and the assumption $u_0 \in H^2(\R^d)$ are merely required to ensure existence of a strong solution and the Duhamel formula~\eqref{eq:Duhamel}.
	Indeed, if one had existence of a strong solution, satisfying the Duhamel formula, the assumptions of Theorems~\ref{1stthm}, \ref{decaynabla}, and \ref{thm:decay_error} could be relaxed to
	\begin{itemize}
		\item
		$B$ globally Lipschitz with Lipschitz constant $L < 1$,
		\item
		$B(0) = 0$,
		\item
		$\lVert (JB)(v) \rVert_F \leq c \lvert v \rvert$ for almost all $v \in \R^d$,
		\item
		$u_0 \in L^1(\R^d) \cap L^2(\R^d)$.
	\end{itemize} 
	The latter condition $u_0 \in L^1(\R^d) \cap L^2(\R^d)$ merits some discussions: While the adaptation for Theorem~\ref{1stthm} is immediate, the proof of Theorem~\ref{decaynabla} uses an integrated version of the energy inequality for $\nabla u$ in Lemma~\ref{Energy inequality 2} which a priori requires $\nabla u_0 \in L^2(\R^d)$, that is $u_0 \in H^1(\R^d)$.
	However, this can be dropped, since an integration of the energy inequality for $u$ in Lemma~\ref{Energy inequality 1} implies in particular $\nabla u(t) \in L^2(\R^d)$ for almost every $t >0$.
	Thus, possibly upon shifting the starting time from $t = 0$ to some time $t \in [0,1]$ for which $\nabla u(t) \in L^2(\R^d)$, which will not affect the decay rate for $t \to \infty$, we would be allowed to assume without loss that $u_0 \in H^1(\R^d)$ in the proof of Theorem~\ref{decaynabla} and thus also the subsequent Theorem~\ref{thm:decay_error}.

We can also propose some explicit examples of admissible nonlinearities, for instance:
\begin{itemize}

\item \[
B(v) = \chi(|v|)\,|v|\,v,
\]
where $\chi:[0,\infty)\to[0,1]$ is a smooth cut-off function satisfying
$\chi(r) = 1$ for $r \leq 1$ and $\chi(r) = 0$ for $r \geq 2$. For instance,
one may take
\[
\chi(r) = 
\begin{cases}
1, & 0 \leq r \leq 1, \\
\exp\left(1 - \dfrac{1}{1 - (r-1)^2}\right), & 1 < r < 2, \\
0, & r \geq 2.
\end{cases}
\]
This function is globally Lipschitz, satisfies $B(0)=0$, behaves quadratically
for small $|v|$ (i.e., $|B(v)| \sim |v|^2$ as $|v|\to 0$), and satisfies the
Jacobian condition $\|JB(v)\|_F \leq c|v|$ for small $|v|$.

\item \[
B(v)=\frac{|v|}{1+|v|}v,
\] which is globally Lipschitz, satisfies $B(0)=0$, and behaves quadratically for small $v$.
\end{itemize}
\end{remark}

\section{Proofs}

\label{sec:proofs}

\subsection{Decay of the solution}

\begin{lemma}[Energy inequality 1] \label{Energy inequality 1}
    Let $u$ be a solution of \eqref{system} and 
    assume that $\operatorname{B}$ is globally Lipschitz continuous with Lipschitz constant $L $ and $B(0) = 0$. Then the following inequality holds
\begin{equation*} 
\frac{1}{2}\frac{\partial}{\partial t}\lVert u \rVert^2_{L^2(\R^d)} +(1-L) \lVert\nabla u \rVert^2_{L^2(\R^d)} \leq 0.
\end{equation*}
 
\end{lemma}

\begin{proof}
   We take $u$ a solution of \eqref{system} that verifies 
    \begin{equation} \label{eqt}
        \partial_tu-\Delta u =  \operatorname{div B}(\nabla u).
    \end{equation}
    Multiplying \eqref{eqt} by $u$ and integrating by parts, we obtain
    \[ u \partial_tu+(\nabla u,\nabla u)=-(\nabla u, \operatorname{B}(\nabla u)).\]
Consequently, 
\[\frac{1}{2}\frac{\partial}{\partial t}\lVert u \rVert^2_{L^2(\R^d)} +\lVert \nabla u \rVert^2_{L^2(\R^d)}=-(\nabla u, \operatorname{B}(\nabla u))_{L^2(\R^d)} \leq L \lVert \nabla u\rVert ^2_{L^2(\R^d)},\]
which implies
\begin{equation*} 
 \frac{1}{2}\frac{\partial}{\partial t}\lVert u \rVert^2_{L^2(\R^d)} +(1-L) \lVert\nabla u \rVert^2_{L^2(\R^d)} \leq 0.
\qedhere
\end{equation*}

\end{proof}

\begin{lemma}[Energy inequality 2] \label{Energy inequality 2}
    Let $u$ be a solution of \eqref{system} and 
    assume that $\operatorname{B}$ is globally Lipschitz continuous with Lipschitz constant $L $ and $B(0) = 0$. Then the following inequality holds
    \begin{equation} 
 \frac{1}{2}\frac{\partial}{\partial t}\lVert \nabla u \rVert^2_{L^2(\R^d)} +(1-L) \lVert\Delta u \rVert^2_{L^2(\R^d)} \leq 0.
\end{equation}
 
\end{lemma}

\begin{proof}

Applying $\nabla $ to \eqref{eqt}, we obtain
    \begin{equation*} 
      \nabla  \partial_t  u-\nabla \cdot \Delta u =\nabla \cdot  \operatorname{div B}(\nabla u).
    \end{equation*}
    
We multiply this by $\nabla u$ in $L^2(\R^d)$ to find
\[
 (\nabla u,  \partial_t \nabla u)_{L^2(\R^d)}- (\nabla u,\nabla \cdot \Delta u)_{L^2(\R^d)}= (\nabla u, \nabla \cdot  \operatorname{div B}(\nabla u))_{L^2(\R^d)}.
\]
Then \[
\frac12\frac{\partial}{\partial t}\norm{\nabla u}_{L^2(\R^d)}^2
+\norm{D^2u}_{L^2(\R^d)}^2
=-\int_{\R^d} D^2u:D(B(\nabla u))\,\mathrm{d}x.
\]
Using the chain rule at points where $B$ is differentiable,
\[
D(B(\nabla u))=JB(\nabla u)D^2u,
\]
and the operator-norm bound $\norm{JB(v)}_{\mathrm{op}}\leq L$, we obtain
\[
\left|\int_{\R^d} D^2u:D(B(\nabla u))\,\mathrm{d}x\right|
\leq L\norm{D^2u}_{L^2(\R^d)}^2.
\]
Hence the energy inequality is
\[
\frac12\frac{\partial}{\partial t}\norm{\nabla u}_{L^2(\R^d)}^2
+(1-L)\norm{D^2u}_{L^2(\R^d)}^2\leq 0.
\]
On $\R^d$, Plancherel's identity gives
\[
\norm{D^2u}_{L^2(\R^d)}^2
=\int_{\R^d}|\xi|^4|\widehat u(\xi)|^2\,\mathrm{d}\xi
=\norm{\Delta u}_{L^2(\R^d)}^2,
\]
so we may replace $\norm{D^2u}_{L^2(\R^d)}$ by $\norm{\Delta u}_{L^2(\R^d)}$, with the coefficient $1-L$.

\end{proof}

\begin{remark} \label{rqenergyinq}
    Note that by integrating the energy inequalities from $0$ to $t$, we find,  for $L<1$,
        \[
    \displaystyle \int_0^t \lVert \nabla u(s) \rVert^2_{L^2(\R^d)}\mathrm{d}s \leq \frac{1}{2(1-L)}\left(\lVert u_0 \rVert^2_{L^2(\R^d)}-\lVert u(t) \rVert^2_{L^2(\R^d)}\right)  \leq \frac{1}{2(1-L)}\lVert u_0 \rVert^2_{L^2(\R^d)}
    \]
   and 
    \[
      \displaystyle \int_0^t \lVert \Delta  u(s) \rVert^2_{L^2(\R^d)}\mathrm{d}s \leq \frac{1}{2(1-L)}\left(\lVert \nabla u_0 \rVert^2_{L^2(\R^d)}-\lVert \nabla u(t) \rVert^2_{L^2(\R^d)}\right)\leq \frac{1}{2(1-L)}\lVert \nabla u_0 \rVert^2_{L^2(\R^d)}.
    \]
    In particular, finiteness of the terms on the right-hand side implies that $\nabla u(t)$ and $\Delta u(t)$ are finite for almost all $t$, cf. the discussion in Remark~\ref{rem:dropping_assumptions}.
\end{remark}

\begin{lemma} \label{divL1norm}
  Assume for the Jacobian of $B$ that 
  $$\lVert (JB)(v) \rVert_{F}=\lvert \operatorname{Tr}[(JB)(v)^*(JB)(v)]\rvert ^{1/2} \leq c \lvert v \rvert,$$ for almost every $ v \in \R^d.$  Under the same conditions as in Lemma \ref{Energy inequality 2}, we have
    \begin{equation*}
        \lVert  \operatorname{div B}(\nabla u) \rVert _{L^1(\R^d) } = \int_{\R^d} \lvert  \operatorname{div B}(\nabla u) (x,t) \rvert  \mathrm{d}x \leq c^2 \lVert \nabla u \rVert^2_{L^2(\R^d)}+ \lVert \Delta u \rVert^2_{L^2(\R^d)}.
    \end{equation*}
\end{lemma}

\begin{proof}

We estimate $ \lVert  \operatorname{div B}(\nabla u) \rVert _{L^1(\R^d) }$ using the chain rule, the Cauchy-Schwarz inequality and the property of the Jacobian of B.
\begin{equation}
    \begin{array}{lcl}
           \int_{\R^d} \lvert  \operatorname{div B}(\nabla u) (x,t) \rvert  \mathrm{d}x  & = &  \displaystyle \int _ {\R ^d}  \lvert\operatorname{Tr} [(JB)(\nabla u) (x,t) .(  \operatorname{Hess}u) (x,t)] \rvert  \mathrm{d}x \\ \\ 
          & \leq & \displaystyle \int _ {\R ^d} \left (\operatorname{Tr} [(JB)(\nabla u) (x,t)^* \, (JB)(\nabla u) (x,t)] \right )^{1/2} \\ \\ 
          & &\left( \operatorname{Tr} [( \operatorname{Hess}u) (x,t)^*( \operatorname{Hess}u) (x,t)]   \right )^{1/2}\mathrm{d}x\\ \\
          & \leq & \displaystyle \int _ {\R ^d} (\operatorname{Tr} [(JB)(\nabla u) (x,t)^* \,   (JB)(\nabla u) (x,t)] \\ \\ & & + \operatorname{Tr} [( \operatorname{Hess}u) (x,t)^*( \operatorname{Hess}u) (x,t)] )   \mathrm{d}x\\ \\
         & \leq &  c^2  \lVert \nabla u \rVert^2_{L^2(\R^d)}+ \lVert \Delta u \rVert^2_{L^2(\R^d)}.
    \end{array}
\end{equation}
\end{proof}
    
\begin{lemma}\label{2ndlemma}
   Let $u$ be a solution of \eqref{system} and 
    assume that $\operatorname{B}$ is globally Lipschitz continuous with Lipschitz constant $L < 1$, $B(0) = 0$, and 
    $$\lVert (JB)(v) \rVert_{F}=\lvert \operatorname{Tr}[(JB)(v)^*(JB)(v)]\rvert ^{1/2} \leq c \lvert v \rvert$$ for almost every $ v \in \R^d$.
    Then, the solution of \eqref{system} satisfies
    \[
      \lvert \hat{u}( \xi ,t)\rvert \leq C_3,
    \]
    where $$C_3=\displaystyle\left(\frac{1}{2 \pi} \right)^{d/2 }   \lVert u_0(x) \rVert_{L^1(\R^d)}+\left( \frac{1}{2 \pi} \right)^{d/2 }\left(  \frac{c^2}{2(1-L)} \lVert u_0 \rVert^2_{L^2(\R^d)}\right)$$$$+\left( \frac{1}{2 \pi} \right)^{d/2 } \left( \frac{1}{2(1-L)}\lVert \nabla  u_0 \rVert^2_{L^2(\R^d)} \right)$$ is a positive constant that depends on the norms of $u_0$ and $\nabla u_0$. 
\end{lemma}
    
\begin{proof}
%
Applying the Fourier transform to the Duhamel formula,~\eqref{eq:Duhamel}, we obtain
\[
 \hat{u}(\xi, t) =\mathrm{e}^{-\lvert \xi \rvert^2t}\hat{u}_0(\xi)+\int_0^t \mathcal{F} \left( \operatorname{div \, B}(\nabla u) \right) \mathrm{e}^{-\lvert \xi \rvert^2 (t-s)} \mathrm{d}s.
\]
We estimate, using Lemma \ref{divL1norm}. 
\begin{equation*}
\begin{array}{ccl}
     \lvert \mathcal{F} \left(\operatorname{div \, B}(\nabla u) \right) \rvert& =&\displaystyle \left\lvert\left( \frac{1}{2 \pi} \right)^{d/2 } \int_{ \R ^d} \mathrm{e}^{-ix \xi } \operatorname{div \, B}(\nabla u) \mathrm{d}x \right\rvert\\ \\
     & \leq & \displaystyle \left( \frac{1}{2 \pi} \right)^{d/2 }  \int_{ \R ^d}\lvert \operatorname{div B}(\nabla u) \rvert\mathrm{d}x 
      \leq  \displaystyle \left( \frac{1}{2 \pi} \right)^{d/2 } \left( c^2\lVert \nabla u \rVert ^2_{L^2(\R^d)} + \lVert \Delta u \rVert ^2_{L^2(\R^d)}\right).
\end{array} 
\end{equation*}
Integrating in time and using the energy inequalities from Lemmas~\ref{Energy inequality 1} and~\ref{Energy inequality 2}, we obtain
\begin{equation*}
    \begin{array}{ll}
      \displaystyle  \int_0^t  \lvert \mathcal{F} \left( \operatorname{div B}(\nabla u) \right) \rvert \mathrm{d}s  & \leq  \displaystyle  \left( \frac{1}{2 \pi} \right)^{d/2 }  \int_0^t \left( c^2 \lVert \nabla u \rVert ^2_{L^2(\R^d)} + \lVert \Delta u \rVert ^2_{L^2(\R^d)} \right) \mathrm{d}s \\  \\
         &  \leq  \displaystyle c^2 \left( \frac{1}{2 \pi} \right)^{d/2 } \frac{1}{2(1-L)} \lVert u_0 \rVert^2_{L^2(\R^d)}+\left( \frac{1}{2 \pi} \right)^{d/2 } \frac{1}{2(1-L)} \lVert \nabla  u_0 \rVert^2_{L^2(\R^d)}.
    \end{array}
\end{equation*}
We can now estimate $\lvert \hat{u}(\xi , t) \rvert$:
\begin{align*}
\lvert \hat{u}(\xi , t) \rvert &\leq  \displaystyle  \lvert \mathrm{e}^{-\lvert \xi \rvert^2t}\hat{u}_0(\xi) \rvert +\left\lvert \int_0^t \mathcal{F} \left( \operatorname{div B}(\nabla u) \right) \mathrm{e}^{-\lvert \xi \rvert^2 (t-s)} \mathrm{d}s\right\rvert \\ 
& \leq  \displaystyle \lvert \hat{u}_0(\xi) \rvert + \int_0^t \left\lvert \mathcal{F} \left( \operatorname{div B}(\nabla u) \right) \mathrm{e}^{-\lvert \xi \rvert^2 (t-s)} \right\rvert \mathrm{d}s \\ 
& \leq  \displaystyle  \left\lvert \left( \frac{1}{2 \pi} \right)^{d/2 } \int_{ \R ^d} \mathrm{e}^{-ix \xi }u_0(x) \mathrm{d}x \right\rvert +\int_0^t \lvert \mathcal{F} \left(\operatorname{div B}(\nabla u) \right)\rvert \mathrm{d}s \\ 
& \leq  \displaystyle  \left(\frac{1}{2 \pi} \right)^{d/2 }  \int_{ \R ^d} \lvert u_0(x) \rvert \mathrm{d}x +c^2\left( \frac{1}{2 \pi} \right)^{d/2 } \frac{1}{2(1-L)} \lVert u_0 \rVert^2_{L^2(\R^d)} +\left( \frac{1}{2 \pi} \right)^{d/2 } \frac{1}{2(1-L)} \lVert \nabla  u_0 \rVert^2_{L^2(\R^d)} \\
&=\left(\frac{1}{2 \pi} \right)^{d/2 }   \lVert u_0 \rVert_{L^1(\R^d)}+\left( \frac{1}{2 \pi} \right)^{d/2 } \frac{1}{2(1-L)}\left( c^2\lVert u_0 \rVert^2_{L^2(\R^d)}+\lVert \nabla  u_0 \rVert^2_{L^2(\R^d)} \right).
\qedhere
\end{align*}  
\end{proof}



We are now ready to prove Theorems~\ref{1stthm} and~\ref{decaynabla}.

\begin{proof}[Proof of Theorem~\ref{1stthm}]
    We start with the first energy inequality, Lemma \ref{Energy inequality 1}
    \begin{equation*} \label{energyinequality2}
\displaystyle \frac{1}{2}\frac{\partial}{\partial t}\lVert u \rVert^2_{L^2(\R ^d)} +(1-L) \lVert\nabla u \rVert^2_{L^2(\R ^d)} \leq 0.
\end{equation*}
Using the Parseval identity, this becomes
\begin{equation} \label{ineaquality}
 \displaystyle \frac{\partial}{\partial t} \int_{ \R ^d} \lvert \hat{u} \rvert^2 \mathrm{d} \xi +2(1-L)  \int_{ \R ^d} \lvert \xi  \rvert^2 \lvert \hat{u} \rvert^2  \mathrm{d}\xi \leq 0.
\end{equation}
Choose a smooth function $\rho (t)$ such that
$$ \rho (0)=1, \quad \rho (t) >0, \quad \rho '(t) >0,
$$
and multiply both sides of the inequality \eqref{ineaquality} by $\rho (t)$ to obtain
\[
     \displaystyle \rho(t) \frac{\partial}{\partial t} \int_{ \R ^d} \lvert \hat{u} \rvert^2 \mathrm{d} \xi +2 \rho(t) (1-L)  \int_{ \R ^d} \lvert \xi  \rvert^2 \lvert \hat{u} \rvert^2  \mathrm{d}\xi \leq 0.
\]
Using integration by parts $$ \displaystyle  \frac{\partial}{\partial t}\left( \rho(t) \int_{ \R ^d} \lvert \hat{u} \rvert^2 \mathrm{d} \xi\right)=\rho(t) \frac{\partial}{\partial t} \int_{ \R ^d} \lvert \hat{u} \rvert^2 \mathrm{d} \xi+\rho'(t)  \int_{ \R ^d} \lvert \hat{u} \rvert^2 \mathrm{d} \xi,$$
the previous inequality can be rewritten as
\begin{equation} \label{inq1}
     \displaystyle \frac{\partial}{\partial t}\left( \rho(t) \int_{ \R ^d} \lvert \hat{u} \rvert^2 \mathrm{d} \xi\right) +2 \rho(t) (1-L)  \int_{ \R ^d} \lvert \xi  \rvert^2 \lvert \hat{u} \rvert^2  \mathrm{d}\xi \leq \rho'(t)  \int_{ \R ^d} \lvert \hat{u} \rvert^2 \mathrm{d} \xi.
\end{equation}
Letting
\[ 
 \sigma(t)=\{ \xi \in \R^d: 2 \rho(t) (1-L)\lvert \xi  \rvert^2 \leq  \rho'(t)  \}
\]
we find
\begin{equation} \label{inq}
    \begin{aligned}
	2 \rho(t) (1-L)  \int_{ \R ^d} \lvert \xi  \rvert^2 \lvert \hat{u} \rvert^2  \mathrm{d}\xi  & \geq   \displaystyle  2 \rho(t) (1-L)  \int_{ \sigma(t) ^c} \lvert \xi  \rvert^2 \lvert \hat{u} \rvert^2  \mathrm{d}\xi   
     \geq \displaystyle \rho'(t)  \int_{ \sigma(t) ^c}\lvert \hat{u} \rvert^2  \mathrm{d}\xi \\  
      & \geq \displaystyle \rho'(t)  \int_{ \R ^d}\lvert \hat{u} \rvert^2  \mathrm{d}\xi -\rho'(t)  \int_{  \sigma(t)}\lvert \hat{u} \rvert^2  \mathrm{d}\xi .  
    \end{aligned}
\end{equation}
Combining \eqref{inq1} with \eqref{inq}, we obtain
\[
     \displaystyle \frac{\partial}{\partial t}\left( \rho(t) \int_{ \R ^d} \lvert \hat{u} \rvert^2 \mathrm{d} \xi\right) \leq \rho'(t)  \int_{  \sigma(t)}\lvert \hat{u} \rvert^2  \mathrm{d}\xi  .
\]
Integrating in time yields
\begin{equation*}
     \displaystyle \rho(t) \int_{ \R ^d} \lvert \hat{u}(\xi,t) \rvert^2 \mathrm{d} \xi - \rho(0) \int_{ \R ^d} \lvert \hat{u}_0 \rvert^2 \mathrm{d} \xi \leq \int_0^t \rho'(s)  \int_{  \sigma(t)}\lvert \hat{u}(\xi, s)  \rvert^2  \mathrm{d}\xi  \mathrm{d} s
\end{equation*}
and, in particular
\begin{equation}\label{I1+I2}
     \displaystyle \rho(t) \int_{ \R ^d} \lvert \hat{u}(\xi,t) \rvert^2 \mathrm{d} \xi\leq \int_{ \R ^d} \lvert \hat{u}_0 \rvert^2 \mathrm{d} \xi+ \int_0^t \rho'(s)  \int_{  \sigma(t)}\lvert \hat{u}(\xi, s)  \rvert^2  \mathrm{d}\xi  \mathrm{d} s =: I_1 + I_2.
\end{equation}
For $I_1$, we use Parseval's identity and $u_0 \in L^2(\R^d)$, to obtain
\begin{equation*}
    I_1 =\int_{ \R ^d} \lvert \hat{u}_0 \rvert^2 \mathrm{d} \xi=\int_{ \R ^d} \lvert u_0 \rvert^2 \mathrm{d} x = \lVert u_0 \rVert^2_{L^2(\R ^d)} .
\end{equation*}
For $I_2$, using Lemma \ref{2ndlemma}, we find
\begin{equation*}
	\begin{aligned}
    I_2 &= \displaystyle\int_0^t \rho'(s)  \int_{  \sigma(t)}\lvert \hat{u}(\xi, s)  \rvert^2  \mathrm{d}\xi  \mathrm{d} s
     \leq  \displaystyle\int_0^t \rho'(s)  \int_{  \sigma(t)}C_3^2 \mathrm{d}\xi  \mathrm{d} s \\
      & \leq \displaystyle C_3^2 \operatorname{vol} S_{d-1} \int_0^t \rho'(s)   \int_{  0}^{\left(\frac{\rho ' (s)}{2(1-L)\rho(s)}\right)^{1/2}} r^{d-1}  \mathrm{d}r \mathrm{d} s 
       \leq \displaystyle \operatorname{vol} S_{d-1} \frac{C_3^2}{d}  \int_0^t \rho'(s)  \left(\frac{\rho ' (s)}{2(1-L)\rho(s)}\right)^{d/2}  \mathrm{d} s \\ 
       & \leq \displaystyle \operatorname{vol} S_{d-1} \frac{C_3^2}{d[2(1-L)]^{d/2}} \int_0^t \rho'(s)  \left(\frac{\rho ' (s)}{\rho(s)}\right)^{d/2}  \mathrm{d} s ,
    \end{aligned}
\end{equation*}
where $\operatorname{vol} S_{d-1}=\frac{2 \pi^{d/2}}{\Gamma(d/2)}$ is the area of the unit sphere in $\R^d$.

Choosing $\rho(t)=(1+t)^d$ yields
\begin{equation*}
    I_2    \leq \displaystyle  \frac{\operatorname{vol} S_{d-1} C_3^2d^{1+d/2}}{d[2(1-L)]^{d/2}}  \int_0^t (1+s)^{d-1} (1+s)^{-d/2} \mathrm{d} s 
       \leq  \frac{2\operatorname{vol} S_{d-1}C_3^2d^{1+d/2}}{d^2[2(1-L)]^{d/2}} (1+t)^{d/2}.
\end{equation*}
Getting back to \eqref{I1+I2}, and using the above estimates on $I_1$ and $I_2$, we obtain
\begin{equation*}
     (1+t)^d \int_{ \R ^d} \lvert \hat{u}(\xi,t) \rvert^2 \mathrm{d} \xi\leq \lVert u_0 \rVert^2_{L^2(\R ^d)} + \displaystyle  \frac{2\operatorname{vol} S_{d-1}C_3^2d^{d/2}}{d[2(1-L)]^{d/2}}  (1+t)^{d/2}.
\end{equation*}
Finally, multiplying both sides by $(1+t)^{-d}$ and using Parseval's identity again, we obtain
\begin{equation*}
    \lVert u(t) \rVert_{L^2(\R ^d)} \leq  \displaystyle \left[ \frac{2\operatorname{vol} S_{d-1}C_3^2d^{d/2}}{d[2(1-L)]^{d/2}} \right]^{1/2} (1+t)^{-d/4}.
\end{equation*}

\end{proof}

    We could now already infer that the error $r(t) =   u(t) - \mathrm{e}^{t \Delta} u_0$ will tend to zero.
    However in light of the fact that both $u(t)$ and $\mathrm{e}^{t \Delta} u_0$ tend to zero, this would not be a very interesting statement.
Therefore, our next goal is to show that the relative error between $u(t)$ and $S(t)u_0$ has a faster decay than that of $u(t)$ and $S(t)u_0$, i.e. 
\[
  \lim_{t \to \infty} \frac{\lVert r(t) \rVert_{L^2(\R ^d)}}{t^{-d/4}}=0.
\]

For this purpose, we shall need an explicit bound on the decay  of $ \lVert \nabla u(t) \rVert_{L^2(\R^d)}$ as stated in Theorem~\ref{decaynabla}.

\begin{proof}[Proof of Theorem~\ref{decaynabla}]
We will adapt the proof of the decay of $u$, to $\nabla u$ using the second energy inequality.

Recall the second energy inequality
    \begin{equation*} 
       \frac{1}{2}\frac{\partial}{\partial t}\lVert \nabla u \rVert^2_{L^2(\R ^d)} +(1-L) \lVert\Delta u \rVert^2_{L^2(\R ^d)} \leq 0.
    \end{equation*}
    from Lemma~\ref{Energy inequality 2}.
Applying the Fourier transform, we find
    \begin{equation*} 
      \displaystyle \frac{\partial}{\partial t} \int_{\R^d}\lvert \xi \rvert ^2 \rvert \hat{u}\lvert ^2\mathrm{d}\xi+2(1-L) \int_{\R^d}\lvert \xi \rvert ^4 \rvert \hat{u}\lvert ^2 \mathrm{d}\xi\leq 0.
    \end{equation*}
Multiply both sides with a smooth function $\rho \colon [0, \infty) \to (0, \infty)$ satisfying
$$
 \rho(0)=1,  \quad \rho (t) >0, \quad \rho '(t) >0,
$$
to obtain
\[
     \displaystyle \rho(t) \frac{\partial}{\partial t} \int_{ \R ^d}\lvert \xi \rvert^2 \lvert \hat{u} \rvert^2 \mathrm{d} \xi +2 \rho(t) (1-L)  \int_{ \R ^d} \lvert \xi  \rvert^4 \lvert \hat{u} \rvert^2  \mathrm{d}\xi \leq 0.
\]
Using that $$ \displaystyle  \frac{\partial}{\partial t}\left( \rho(t) \int_{ \R ^d}\lvert \xi \rvert^2  \lvert \hat{u} \rvert^2 \mathrm{d} \xi\right)=\rho(t) \frac{\partial}{\partial t} \int_{ \R ^d}\lvert \xi \rvert^2  \lvert \hat{u} \rvert^2 \mathrm{d} \xi+\rho'(t)  \int_{ \R ^d} \lvert \xi \rvert^2 \lvert \hat{u} \rvert^2 \mathrm{d} \xi,$$
we can rewrite this as
\begin{equation} \label{inq1GRAD}
     \displaystyle \frac{\partial}{\partial t}\left( \rho(t) \int_{ \R ^d} \lvert \xi \rvert^2  \lvert \hat{u} \rvert^2 \mathrm{d} \xi\right) +2 \rho(t) (1-L)  \int_{ \R ^d} \lvert \xi  \rvert^4 \lvert \hat{u} \rvert^2  \mathrm{d}\xi \leq \rho'(t)  \int_{ \R ^d} \lvert \xi \rvert^2  \lvert \hat{u} \rvert^2 \mathrm{d} \xi.
\end{equation}
Defining
\[ 
 \sigma(t)=\{ \xi \in \R^d: 2 \rho(t) (1-L)\lvert \xi  \rvert^2 \leq  \rho'(t)  \}
\]
we find
\begin{equation} \label{inqGRAD}
    \begin{aligned}
    \displaystyle     2 \rho(t) (1-L)  \int_{ \R ^d} \lvert \xi  \rvert^4 \lvert \hat{u} \rvert^2  \mathrm{d}\xi  & \geq   \displaystyle  2 \rho(t) (1-L)  \int_{ \sigma(t) ^c} \lvert \xi  \rvert^2  \lvert \xi  \rvert^2  \lvert \hat{u} \rvert^2  \mathrm{d}\xi  \\ 
     \geq \displaystyle \rho'(t)  \int_{ \sigma(t) ^c} \lvert \xi  \rvert^2  \lvert \hat{u} \rvert^2  \mathrm{d}\xi
     & \geq \displaystyle \rho'(t)  \int_{ \R ^d} \lvert \xi  \rvert^2  \lvert \hat{u} \rvert^2  \mathrm{d}\xi -\rho'(t)  \int_{  \sigma(t)} \lvert \xi  \rvert^2  \lvert \hat{u} \rvert^2  \mathrm{d}\xi .
    \end{aligned}
\end{equation}
Combining \eqref{inqGRAD} with \eqref{inq1GRAD}, we obtain
\[
     \displaystyle \frac{\partial}{\partial t}\left( \rho(t) \int_{ \R ^d}  \lvert \xi  \rvert^2 \lvert \hat{u} \rvert^2 \mathrm{d} \xi\right) \leq \rho'(t)  \int_{  \sigma(t)} \lvert \xi  \rvert^2  \lvert \hat{u} \rvert^2  \mathrm{d}\xi  .
\]
Integrating in time yields
\begin{equation*}
     \displaystyle \rho(t) \int_{ \R ^d} \lvert \xi  \rvert^2  \lvert \hat{u}(\xi,t) \rvert^2 \mathrm{d} \xi - \rho(0) \int_{ \R ^d}  \lvert \xi  \rvert^2  \lvert \hat{u}_0 \rvert^2 \mathrm{d} \xi \leq \int_0^t \rho'(s)  \int_{  \sigma(t)} \lvert \xi  \rvert^2  \lvert \hat{u}(\xi, s)  \rvert^2  \mathrm{d}\xi  \mathrm{d} s,
\end{equation*}
and, in particular
\begin{equation}\label{I1+I2GRAD}
     \displaystyle \rho(t) \int_{ \R ^d} \lvert \xi  \rvert^2 \lvert \hat{u}(\xi,t) \rvert^2 \mathrm{d} \xi\leq \int_{ \R ^d} \lvert \xi  \rvert^2  \lvert \hat{u}_0 \rvert^2 \mathrm{d} \xi+ \int_0^t \rho'(s)  \int_{  \sigma(t)} \lvert \xi  \rvert^2  \lvert \hat{u}(\xi, s)  \rvert^2  \mathrm{d}\xi  \mathrm{d} s =: I_1 + I_2.
\end{equation}
For $I_1$, we use Parseval's identity and the fact that $\nabla u_0 \in L^2(\R^d)$ to obtain
\begin{equation}
	\label{eq:estimate_I_1_grad}
    I_1 =\int_{ \R ^d}  \lvert \xi  \rvert^2 \lvert \hat{u}_0 \rvert^2 \mathrm{d} \xi= \lVert \nabla u_0 \rVert^2_{L^2(\R ^d)} .
\end{equation}
For $I_2$, Lemma~\ref{2ndlemma} yields
\begin{equation*}
\begin{aligned}
    I_2 &= \displaystyle\int_0^t \rho'(s)  \int_{  \sigma(t)} \lvert \xi  \rvert^2  \lvert \hat{u}(\xi, s)  \rvert^2  \mathrm{d}\xi  \mathrm{d} s  
     \leq  \displaystyle\int_0^t \rho'(s)  \frac{\rho ' (s)}{2(1-L)\rho(s)} \int_{  \sigma(t)}C_3^2 \mathrm{d}\xi  \mathrm{d} s \\ 
      & \leq  \displaystyle \operatorname{vol} S_{d-1} C_3^2 \int_0^t \rho'(s)  \frac{\rho ' (s)}{2(1-L)\rho(s)}  \int_{  0}^{\left(\frac{\rho ' (s)}{2(1-dL)\rho(s)}\right)^{1/2}} r^{d-1}\mathrm{d}r \mathrm{d} s \\
       & \leq  \displaystyle \frac{\operatorname{vol} S_{d-1}C_3^2}{d}  \int_0^t \rho'(s)  \frac{\rho ' (s)}{2(1-L)\rho(s)}  \left(\frac{\rho ' (s)}{2(1-L)\rho(s)}\right)^{d/2}  \mathrm{d} s \\ 
       & \leq  \displaystyle \frac{\operatorname{vol} S_{d-1}C_3^2}{d[2(1-L)]^{1+d/2}} \int_0^t \rho'(s)  \frac{\rho ' (s)}{\rho(s)} \left(\frac{\rho ' (s)}{\rho(s)}\right)^{d/2}  \mathrm{d} s .
    \end{aligned}
\end{equation*}
Choosing $\rho(t)=(1+t)^{d+1}$ we find
\begin{equation}
	\label{eq:estimate_I_2_grad}
    I_2     \leq  \displaystyle \frac{\operatorname{vol} S_{d-1}C_3^2(d+1)^{2+d/2}}{d[2(1-L)]^{1+d/2}}   \int_0^t (1+s)^{d-1} (1+s)^{-d/2} \mathrm{d} s \\ 
       \leq  \displaystyle \frac{2\operatorname{vol} S_{d-1}C_3^2(d+1)^{2+d/2}}{d^2[2(1-L)]^{1+d/2}}  (1+t)^{d/2}.
\end{equation}
Combining~\eqref{eq:estimate_I_1_grad} and~\eqref{eq:estimate_I_2_grad} with~\eqref{I1+I2GRAD}, we find
\begin{equation*}
     \displaystyle (1+t)^{d+1} \int_{ \R ^d}  \lvert \xi \rvert ^2\lvert \hat{u}(\xi,t) \rvert^2 \mathrm{d} \xi\leq \lVert \nabla u_0 \rVert^2_{L^2(\R ^d)} + \displaystyle \frac{2\operatorname{vol} S_{d-1}C_3^2(d+1)^{2+d/2}}{d^2[2(1-L)]^{1+d/2}} (1+t)^{d/2}.
\end{equation*}
Finally, multiplying both sides by $(1+t)^{-d-1}$ and using Parseval's identity again implies 
\begin{equation*}
    \lVert \nabla u(t) \rVert_{L^2(\R ^d)} \leq  \displaystyle \left[\frac{2\operatorname{vol} S_{d-1}C_3^2(d+1)^{2+d/2}}{d^2[2(1-L)]^{1+d/2}} \right]^{1/2} (1+t)^{-d/4-1/2}.
\end{equation*}

\end{proof}

\subsection{Decay of the error}

\bigskip

In this section, we prove faster decay of the error term $$ r(t)= \int_0^{t}  S(t-s)  \operatorname{div B} (\nabla u(s))  \mathrm{d}s.$$

\subsubsection{Small time integral}

\begin{lemma}
    Under the same conditions as in Theorem \ref{thm:decay_error}, we have
    \begin{equation}\label{smalltime}
    \int_0^{t-t^{\alpha}} \lVert S(t-s)  \operatorname{div B} (\nabla u(s)) \rVert_{L^2(\R ^d)} \mathrm{d}s  \leq  
     \frac{Cc\lVert u_0 \rVert ^2_{L^2(\R ^d)}}{2(1-L)} 
    t^{\alpha (-1/2-d/4)}.
\end{equation}
\end{lemma}
\begin{proof}

To estimate the integral from $0$ to $t-t^{\alpha}$, we use the property of the free heat semi-group,
\[
 \lVert \nabla S(t)v\rVert_{L^2(\R^d)} \leq C t^{-1/2-d/4} \lVert v \rVert_{L^1(\R^d)}.
\]
We have,
\begin{equation*}
    \begin{array}{ccl}
        \lVert  \int_0^{t-t^{\alpha}}S(t-s)  \operatorname{div B} (\nabla u(s))  \mathrm{d}s\rVert_{L^2(\R ^d)}
         & \leq & \lVert \int_0^{t-t^{\alpha}} \nabla S(t-s)  \operatorname{B} (\nabla u(s))  \mathrm{d}s\rVert_{L^2(\R ^d)} \\ \\
         & \leq & \int_0^{t-t^{\alpha}} \lVert \nabla S(t-s)  \operatorname{B} (\nabla u(s)) \rVert_{L^2(\R ^d)} \mathrm{d}s \\ \\
         & \leq & C \int_0^{t-t^{\alpha}}  (t-s)^{-1/2-d/4} \lVert \operatorname{B} (\nabla u(s)) \rVert_{L^1(\R^d)} \mathrm{d}s \\ \\
              &   \leq & C\max_{s \in(0, t-t^{\alpha })}(t-s)^{-1/2-d/4}  \int_0^{t-t^{\alpha}}  \lVert \operatorname{ B} (\nabla u(s)) \rVert_{L^1(\R^d)} \, \mathrm{d}s. \\ \\
             &   \leq & C(t^{\alpha})^{-1/2-d/4}  \int_0^{t-t^{\alpha}}  \lVert \operatorname{ B} (\nabla u(s)) \rVert_{L^1(\R^d)} \, \mathrm{d}s. \\ \\
    \end{array}
\end{equation*}
We will show that $\int_0^{t-t^{\alpha}}  \lVert \operatorname{ B} (\nabla u(s)) \rVert_{L^1(\R^d)} \, \mathrm{d}s$ is bounded using that $\lvert \operatorname{ B} (\nabla u(s)) \rvert  \leq c \lvert \nabla u(s)\rvert^2 $, cf. Remark~\ref{rem:Jacobian_implies_quadratic}, and the first energy inequality. 
For this purpose, we estimate
\begin{equation*}
    \begin{array}{ccl}
    \displaystyle  \int_0^{t-t^{\alpha}}  \lVert \operatorname{ B} (\nabla u(s)) \rVert_{L^1(\R^d)} \, \mathrm{d}s & = & \displaystyle \int_0^{t-t^{\alpha}} \int_{\R^d} \lvert \operatorname{ B} (\nabla u(s)) \rvert \, \mathrm{d}s  \\ \\
       & \leq & \displaystyle c \int_0^{t-t^{\alpha}} \int_{\R^d} \lvert \nabla u(s)\rvert^2 \, \mathrm{d}s  \\ \\
       & \leq & \displaystyle  c \int_0^{t-t^{\alpha}}  \lVert \nabla u(s)\rVert^2_{L^2(\R^d)} \, \mathrm{d}s  \\ \\
       & \leq & \displaystyle \frac{c}{2(1-L)} \lVert u_0 \rVert ^2_{L^2(\R ^d)}.
    \end{array}
\end{equation*}
In particular, this implies
\begin{equation*}
    \int_0^{t-t^{\alpha}} \lVert S(t-s)  \operatorname{div B} (\nabla u(s)) \rVert_{L^2(\R ^d)} \mathrm{d}s  \leq  \frac{c \, C\lVert u_0 \rVert ^2_{L^2(\R ^d)}}{2(1-L)} t^{\alpha (-1/2-d/4)}.
    \qedhere
\end{equation*}
\end{proof}
\subsubsection{Large time integral}

\begin{lemma}
    Under the same conditions as in Theorem \ref{thm:decay_error}, we have
    \begin{equation}\label{largetime}
        \int_{t-t^{\alpha}}^t \lVert S(t-s)  \operatorname{div B} (\nabla u(s)) \rVert_{L^2(\R ^d)} \mathrm{d}s \lesssim t  ^{-d/4-1/2+\alpha /2}.
    \end{equation}
\end{lemma}

\begin{proof}
We start by integrating by parts, using the free heat semi-group estimate 
\[
 \lVert \nabla S(t)v\rVert_{L^2(\R^d)} \leq C t^{-1/2} \lVert v \rVert_{L^2(\R^d)}.
\]
Then, we use that B is globally Lipschitz and the decay rate of $\lVert \nabla u \rVert_{L^2(\R^d)}$, proved in Theorem \ref{decaynabla}.
  \begin{equation*}
      \begin{array}{lcl}
       \int_{t-t^{\alpha}}^t \lVert S(t-s)  \operatorname{div B} (\nabla u(s)) \rVert_{L^2(\R ^d)} \mathrm{d}s     &  \leq & C \int_{t-t^{\alpha}}^t  (t-s)^{-1/2}  \lVert\operatorname{ B} (\nabla u(s)) \rVert_{L^2(\R ^d)} \mathrm{d}s  \\ \\
          &  \leq &  CL\int_{t-t^{\alpha}}^t  (t-s)^{-1/2}  \lVert \nabla u(s) \rVert_{L^2(\R ^d)} \mathrm{d}s  \\ \\
           &  \leq &  CL C_2\int_{t-t^{\alpha}}^t  (t-s)^{-1/2} (1+s)^{-d/4-1/2} \mathrm{d}s  \\ \\
            &  \leq & CLC_2 \max_{s\in (t-t^{\alpha}, t)} 
            (1+s)^{-d/4-1/2} \int_{t-t^{\alpha}}^t  (t-s)^{-1/2} \mathrm{d}s  \\ \\
              &  \leq & 2 CLC_2   \left(1+t-t^{\alpha }\right)^{-d/4-1/2} \left(t^{\alpha }\right)^{1/2}  \\ \\
               &  \leq &  2 CLC_2   \left(\frac{t}{2}\right)^{-d/4-1/2} \left(t^{\alpha }\right)^{1/2} \\ \\
                 & = & \displaystyle  2^{d/4+3/2}CLC_2  t^{-d/4-1/2+\alpha /2}.       \qedhere
      \end{array}
  \end{equation*}   
\end{proof}

We can now prove Theorem~\ref{thm:decay_error}.

\begin{proof}[Proof of Theorem~\ref{thm:decay_error}]
    We split the integral from $0$ to $t$ at $t-t^{\alpha}$ for some $\alpha \in (0,1)$ to be determined, that is
\begin{equation*}
    \begin{array}{ccl}
       \lVert r(t) \rVert_{L^2(\R ^d)}  & \leq & \lVert\int_0^{t-t^{\alpha}}  S(t-s)  \operatorname{div B} (\nabla u(s))  \mathrm{d}s \rVert_{L^2(\R ^d)}+ \lVert \int_{t-t^{\alpha}}^t S(t-s)  \operatorname{div B} (\nabla u(s))  \mathrm{d}s\rVert_{L^2(\R ^d)} \\ \\
          & \leq & \displaystyle  \frac{Cc\lVert u_0 \rVert ^2_{L^2(\R ^d)}}{2(1-L)} 
    t^{\alpha (-1/2-d/4)} + 2^{d/4+3/2}CLC_2  t^{-d/4-1/2+\alpha /2}.
\end{array}
\end{equation*}
We now choose $\alpha$ such that the estimates \eqref{smalltime} and \eqref{largetime} both yield a better decay than that of $u$, that is, 
\[
 \alpha (-1/2-d/4)<-d/4
\]
and
\[
 -d/4-1/2+\alpha/2<-d/4.
\]
Indeed, these two conditions can be satisfied for $\alpha \in \left ( \frac{d}{d+2}, 1 \right)$ and the fastest decay is obtained in case $\alpha = \frac{d+2}{d+4}$ in which case both terms decay proportionally to $t^{- d/4 - 1/(d+4)}$.
This concludes the proof of Theorem~\ref{thm:decay_error}.
\qedhere

\end{proof}

\section*{Data availability statement}

No datasets were generated or analyzed during the current study. All derivations/equations are included in the article.


\end{document}